\documentclass[11pt]{amsart}
\usepackage{amsmath,amssymb,colonequals,verbatim}

\numberwithin{equation}{section}

\newtheorem{thm}[equation]{Theorem}

\newtheorem{lemma}[equation]{Lemma}
\newtheorem{cor}[equation]{Corollary}

\theoremstyle{definition}
\newtheorem{rmk}[equation]{Remark}

\newcommand{\F}{\mathbb{F}}
\newcommand{\bP}{\mathbb{P}}

\DeclareMathOperator{\ord}{ord}

\DeclareMathOperator{\Tr}{Tr}
\DeclareMathOperator{\PSL}{PSL}
\newcommand{\mybar}[1]{#1\llap{$\overline{\phantom{\rm#1}}$}}

\usepackage[colorlinks,pagebackref,pdftex,bookmarks=false]{hyperref}

\begin{document}

\title[Permutation polynomials over $\F_{q^3}$]
{A new construction of permutation polynomials over $\F_{q^3}$}

\author{Zhiguo Ding}
\address{
  School of Mathematics and Statistics, 
  Hunan Normal University, 
  Changsha 410081, China
}
\email{ding8191@qq.com}

\author{Xu Song}
\address{
  School of Mathematical Science, 
  Hebei Normal University, 
  Shijiazhuang 050024, China
}
\email{songxu@amss.ac.cn}

\author{Wei Xiong}
\address{
  School of Mathematics,
  Hunan University,
  Changsha 410082, China
}
\email{weixiong@amss.ac.cn}

\date{\today}

\thanks{
The first author thanks Michael Zieve for introducing him to this subject.
The second author was supported by Doctor Foundation of Hebei Normal University (No.\ L2024B02).
The third author was supported by National Natural Science Foundation of China (No.\ 12571010).
The authors thank Ye Tian for his help and encouragement, which in particular lead to this collaboration. 
}

\begin{abstract} 
We determine all permutation polynomials among several families of polynomials over $\F_{q^3}$ for arbitrary prime powers $q$. We obtain some new families of permutation polynomials over $\F_{q^3}$ with simple coefficients for infinitely many characteristics. As a specific consequence, our results resolve the generalization of conjectures of Zhang, Zheng, Wang, Peng, and Li in the even characteristic. Our proofs are conceptually short and involve no complicated computations, in contrast to the proofs of results on permutation polynomials which were published previously. Moreover, we develop a totally new systematic method in this paper for the study of permutation polynomials.
\end{abstract}

\maketitle


\section{Introduction}

A polynomial $f(X)\in\F_q[X]$ is called a \emph{permutation polynomial} if the function $\alpha\mapsto f(\alpha)$ induces a permutation of $\F_q$. Permutation polynomials have long been studied, both for their intrinsic interest and due to applications to various areas of computer science and engineering. 

The paper \cite{Z-Redei} introduced a systematic method to construct permutation polynomials over fields $\F_{q^2}$ of square order, by showing that the question of whether a polynomial of the form $X^r B(X^{q-1})$ permutes $\F_{q^2}$ may be reduced to the question of whether an associated rational function permutes $\bP^1(\F_q)\colonequals\F_q\cup\{\infty\}$. In over $100$ papers by various authors, the above method has been used to produce specific instances of permutation polynomials over $\F_{q^2}$ with few terms, for a typical example see \cite{DZ-quad}. But there are much few permutation polynomials over fields $\F_{q^3}$ of cubic order in the literature, all of which are constructed by using variants of the method introduced in \cite{Z-Redei}.

In this paper, we will introduce a new method to study permutation polynomials over $\F_{q^3}$. Most constructions of permutation polynomials in the literature use one step of reduction, which is some variant of the above described construction initially introduced in \cite{Z-Redei}, which is either multiplicative or additive. Our construction of permutation polynomials over $\F_{q^3}$ in this paper is different, in the sense that it consists of two steps of reduction. The first step reduces permutations on $\F_{q^3}$ to permutations on the set of roots of $X^{q^2}+X^q+X$, and the second step reduces permutations on the set of roots of $X^{q^2}+X^q+X$ to permutations on the set of roots of $X^{q+1}+X+1$. It is interesting to note that the first step is of additive flavor while the second step is of multiplicative flavor. It is surprising that the mixture of such two reductions of different flavors can lead to permutation polynomials over $\F_{q^3}$ with at most seven terms and very simple coefficients.

As an illustration of our new approach, we will construct several families of permutation polynomials over $\F_{q^3}$ with at most seven terms, among which one family has only three terms and three family have five terms, all of them have coefficients $1$. More importantly, we will construct two families of permutation polynomials over $\F_{q^3}$ for infinitely many characteristics. Our approach can apply to the study of other families of polynomials over $\F_{q^3}$, but only permutation polynomials over $\F_{q^3}$ with simple expressions are interesting, so we focus on such permutation polynomials. 

Our main results are as follows. For any nonzero integer $n$, we write $\ord_2(n)$ for the largest integer $s\ge 0$ such that $2^s$ divides $n$. We also write $\ord_2(0)=\infty$. In other words, $\ord_2(n)$ is the discrete $2$-adic valuation of $n$.

\begin{thm}\label{main1}
Assume $q\colonequals p^k$, $Q\colonequals p^{\ell}$, $R\colonequals p^m$, and $S\colonequals p^n$ for a prime $p$ and some integers $k,\ell,m,n$ with $k>0$ and $\ell,m,n\ge 0$. Suppose $c\in\F_q^*$ and write
\[
f(X)\colonequals X^{Q+1} \circ (X^q-X) + (X^{q^2}+X^q+X) \circ cX^{R+S}. 
\]
Then $f(X)$ permutes\/ $\F_{q^3}$ if and only if $\gcd\bigl(q-1,(Q+1)(R+S)\bigr)=1$, or equivalently, $p=2$ and $\ord_2(k)\le\min\bigl(\ord_2(\ell),\ord_2(m-n)\bigr)$.
\end{thm}

The permutation polynomials in Theorem~\ref{main1} have at most seven terms. By some special choices of $c$ and $Q,R,S$ in Theorem~\ref{main1}, we can obtain as follows one family of permutation trinomials and three families of permutation pentanomials over $\F_{q^3}$ with coefficients $1$. 

\begin{cor}\label{main2}
Assume $q\colonequals p^k$, $Q\colonequals p^{\ell}$ for a prime $p$ and some integers $k,\ell$ with $k>0$ and $\ell\ge 0$. Then the following are equivalent:
\begin{enumerate}
\item $\gcd(q-1,Q+1)=1$, or equivalently $p=2$ and $\ord_2(k)\le\ord_2(\ell)$;
\item $f_1(X)\colonequals X^{q^2Q+q}+X^{qQ+q^2}+X^{Q+1}$ permutes\/ $\F_{q^3}$;
\item $f_2(X)\colonequals X^{q^2Q+1}+X^{qQ+1}+X^{Q+q^2}+X^{Q+q}-X^{Q+1}$ permutes\/ $\F_{q^3}$;
\item $f_3(X)\colonequals X^{q^2Q+q}+X^{qQ+q}+X^{Q+q^2}-X^{Q+q}+X^{Q+1}$ permutes\/ $\F_{q^3}$;
\item $f_4(X)\colonequals X^{q^2Q+q^2}+X^{qQ+q^2}-X^{Q+q^2}+X^{Q+q}+X^{Q+1}$ permutes\/ $\F_{q^3}$.
\end{enumerate}
\end{cor}

\begin{rmk}
The permutation pentanomials $f_2(X)$ and $f_3(X)$ in Corollary~\ref{main2} solve two conjectures in \cite{ZZWPL}, and the others are new in the literature as far as our knowledge. In particular, the permutation trinomial $f_1(X)$ looks like more interesting. It is very surprising to obtain a classification of permutation trinomials over $\F_{q^3}$ along a long way of construction. 
\end{rmk}

\begin{thm}\label{main3}
Assume $q\colonequals p^k$, $Q\colonequals p^{\ell}$, and $R\colonequals p^m$ for a prime $p$ and some integers $k,\ell,m$ with $k>0$ and $\ell,m\ge 0$. Suppose $c\in\F_{q^3}$ and write
\[
f(X)\colonequals X^{Q+1} \circ (X^q-X) + (X^{q^2}+X^q+X) \circ cX^R.
\]
Then $f(X)$ permutes\/ $\F_{q^3}$ if and only if $\gcd(q-1,Q+1)=1$ and\/ $\Tr_{\F_{q^3}/\F_q}(c)\ne 0$, or equivalently, $p=2$, $\ord_2(k)\le\ord_2(\ell)$, and $c^{q^2}+c^q+c\ne 0$.
\end{thm}

The permutation polynomials in Theorem~\ref{main3} have at most seven terms. By simply taking $c=1$ in Theorem~\ref{main3}, we can obtain as follows a family of permutation polynomials over $\F_{q^3}$ with coefficients $1$. 

\begin{cor}\label{main4}
Assume $q\colonequals p^k$, $Q\colonequals p^{\ell}$, and $R\colonequals p^m$ for a prime $p$ and some integers $k,\ell,m$ with $k>0$ and $\ell,m\ge 0$. Write
\[
f_5(X)\colonequals X^{qQ+q}-X^{qQ+1}-X^{Q+q}+X^{Q+1}+X^{q^2R}+X^{qR}+X^{R}.
\]
Then $f_5(X)$ permutes\/ $\F_{q^3}$ if and only if $\gcd(q-1,Q+1)=1$, or equivalently $p=2$ and $\ord_2(k)\le\ord_2(\ell)$.
\end{cor}

The key part of the proofs of the above results consists of the proof of the following result, which is of independent interest.

\begin{thm}\label{key}
Suppose $q=p^k$ and $Q=p^{\ell}$ for a prime $p$ and some integers $k>0$ and $\ell\ge 0$. Then $(X^q-X)\circ X^{Q+1}$ permutes the set of roots in\/ $\F_{q^3}$ of $X^{q^2}+X^q+X$ if and only if $\gcd(q-1,Q+1)=1$, or equivalently $p=2$ and $\ord_2(k)\le\ord_2(\ell)$.   
\end{thm}

Variants of Theorem~\ref{key} have appeared in several papers in different forms, most notably in the papers \cite{DZ-quad} and \cite{Gologlu-fractional} which determine all permutation rational functions in certain families of rational functions over $\F_q$.

The following is a variant of Theorem~\ref{key} with $X^{Q+1}$ replaced by $X^3$.

\begin{thm}\label{key-deg3}
The polynomial $(X^q-X)\circ X^3$ permutes the set of roots in\/ $\F_{q^3}$ of $X^{q^2}+X^q+X$ if and only if $q\equiv 2\pmod 3$.   
\end{thm}

It is interesting to note that Theorem~\ref{key-deg3} gives rise to permutation on roots of $X^{q^2}+X^q+X$ for infinitely many characteristics. Consequently, Theorem~\ref{key-deg3} can be used to obtain the following family of permutation polynomials over $\F_{q^3}$ for infinitely many characteristics.

\begin{thm}\label{main-deg3}
Assume $q\colonequals p^k$ and $Q\colonequals p^{\ell}$ for a prime $p$ and some integers $k,\ell$ with $k>0$ and $\ell\ge 0$. Suppose $c\in\F_{q^3}$ and write
\[
f(X)\colonequals X^3 \circ (X^q-X) + (X^{q^2}+X^q+X) \circ cX^Q.
\]
Then $f(X)$ permutes\/ $\F_{q^3}$ if and only if $q\equiv 2\pmod 3$ and $c^{q^2}+c^q+c\ne 0$.
\end{thm}

The permutation polynomials in Theorem~\ref{main-deg3} have at most seven terms. By simply taking $c\in\F_q^*$ in Theorem~\ref{main-deg3}, we can obtain the following family of permutation polynomials over $\F_{q^3}$ with coefficients in $\F_q$. 

\begin{cor}\label{main2-deg3}
Assume $q\colonequals p^k$ and $Q\colonequals p^{\ell}$ for a prime $p$ and some integers $k,\ell$ with $k>0$ and $\ell\ge 0$. For any $c\in\F_q^*$, write
\[
f_6(X)\colonequals X^{3q}-3X^{2q+1}+3X^{q+2}-X^3+cX^{q^2Q}+cX^{qQ}+cX^{Q}.
\]
Then $f_6(X)$ permutes\/ $\F_{q^3}$ if and only if $q\equiv 2\pmod3$.
\end{cor}

\begin{rmk}
By taking $c=1$ and $Q=1$, we get a bivariate degree-$3$ polynomial $f(X,Y)\colonequals (Y-X)^3+(Y^2+Y+X)$ over the prime field $\F_p$ with $p\equiv 2\pmod 3$, satisfying $f(X,X^q)$ is a permutation polynomial over $\F_{q^3}$ for any $q=p^k$ with an odd integer $k>0$. This might be viewed as a generalization of classical exceptional polynomials over $\F_p$, which permute $\F_{p^m}$ for infinitely many integers $m>0$. See \cite{DXZ,GRZ,GZ,Z-exc} for the study of exceptional polynomials and its connection to permutation polynomials. 
\end{rmk}

This paper is organized as follows. In the next section we recall some background material and prove some simple results. In Section~\ref{proof} we give proofs of our main results listed above.


\section{Background results}

In this section we recall some known results which we will use in our proofs. Throughout this paper, we will use the following notation:

\begin{itemize}
\item $p$ is a prime;
\item $q=p^k$ for some integer $k>0$;
\item $Q=p^{\ell}$ for some integer $\ell\ge 0$;
\item $\mybar\F_q$ is the algebraic closure of $\F_q$;
\item for any subfield $K$ of $\mybar\F_q$, we write $\bP^1(K)\colonequals K\cup\{\infty\}$, which is the set of $\F_q$-points of the projective line $\bP^1$ over $K$;
\item for any integer $m>0$, let $\mu_m$ be the set of $m$-th roots of unity in $\mybar\F_q$;
\item $\Gamma$ is the set of roots of $X^{q^2}+X^q+X$ in $\F_{q^3}$;
\item $\Gamma^*$ is the set of nonzero roots of $X^{q^2}+X^q+X$ in $\F_{q^3}$;
\item $\Lambda$ is the set of roots of $X^{q+1}+X+1$ in $\F_{q^3}$;
\item for any nonzero integer $n$, we write $\ord_2(n)$ for the largest integer $s\ge 0$ such that $2^s$ divides $n$, and we also write $\ord_2(0)=\infty$;
\item for any $g(X)\in\mybar\F_q(X)$, we write $g^{(q)}(X)$ for the rational function obtained from $g(X)$ by raising its coefficients to their $q$-th powers.
\end{itemize}

The following result is well-known and easy, cf.\ e.g.\ \cite[Lemma~5.2]{DZ-quad}:

\begin{lemma}\label{gcd}
If $k$ and $\ell$ are integers with $k>0\le\ell$ then
\begin{itemize}
\item $\gcd(2^k+1,2^\ell+1)=1$ if and only if $\ord_2(k)\ne\ord_2(\ell)$;
\item $\gcd(2^k-1,2^\ell+1)=1$ if and only if $\ord_2(k)\le\ord_2(\ell)$.
\end{itemize}
\end{lemma}

We recall some results about degree-$1$ rational functions from \cite{Z-Redei}:

\begin{lemma}\label{mu}
A degree-$1$ $\rho(X)\in\mybar\F_q(X)$ permutes $\mu_{q+1}$ if and only if $\rho(X)=(b^qX+a^q)/(aX+b)$ for some $a,b\in\F_{q^2}$ with $a^{q+1}\ne b^{q+1}$.
\end{lemma}

\begin{lemma}\label{Ftomu}
A degree-$1$ $\rho(X)\in\mybar\F_q(X)$ maps\/ $\bP^1(\F_q)$ onto $\mu_{q+1}$ if and only if $\rho(X)=(a^qX+b^q)/(aX+b)$ for some $a,b\in\F_{q^2}$ with $a^qb\ne ab^q$.
\end{lemma}

\begin{lemma}\label{mutoF}
A degree-$1$ $\rho(X)\in\mybar\F_q(X)$ maps $\mu_{q+1}$ onto\/ $\bP^1(\F_q)$ if and only if $\rho(X)=(bX+b^q)/(aX+a^q)$ for some $a,b\in\F_{q^2}$ with $a^qb\ne ab^q$.
\end{lemma}

As a direct consequence of Lemmas~\ref{mu}, \ref{Ftomu}, and \ref{mutoF}, we have 

\begin{cor}\label{deg1}
Assume that $q$ is a prime power with $\gcd(3,q)=1$. Suppose $\rho(X)\colonequals (\omega X+1)/(X+\omega)$ for some $\omega\in\mybar\F_q$ with $\omega^2+\omega+1=0$. Then
\begin{itemize}
\item if $q\equiv 1\pmod 3$ then $\rho(X)$ permutes both\/ $\bP^1(\F_q)$ and $\mu_{q+1}$;
\item if $q\equiv 2\pmod 3$ then $\rho(X)$ swaps\/ $\bP^1(\F_q)$ and $\mu_{q+1}$, in the sense that $\rho(X)$ maps\/ $\bP^1(\F_q)$ onto $\mu_{q+1}$ and it maps $\mu_{q+1}$ onto\/ $\bP^1(\F_q)$; 
\item if $q$ is even then $\rho(X)$ is an involution, i.e.\ $\rho^{-1}(X)=\rho(X)$.
\end{itemize}  
\end{cor}


\section{Proofs of the main results}\label{proof}

In this section we show the main results in this paper. The following result can be viewed an additive analogue of \cite[Lemma~2.1]{Zlem} and \cite[Lemma~2.2]{Z-Redei}:

\begin{lemma}\label{step1}
Assume $q\colonequals p^k$ for a prime $p$ and an integer $k>0$. Suppose 
\[
f(X)\colonequals L(X) + B(X) \circ (X^q-X) + (X^{q^2}+X^q+X) \circ C(X) 
\]
for some $B(X), C(X) \in \F_{q^3}[X]$ and some $L(X)\colonequals\sum_{i=0}^r a_iX^{p^i}\in\F_q[X]$. Then $f(X)$ permutes\/ $\F_{q^3}$ if and only if both of the following hold:
\begin{itemize}
\item $L(X) + (X^q-X)\circ B(X)$ permutes the set of roots of $X^{q^2}+X^q+X$;
\item $L(X) + (X^{q^2}+X^q+X) \circ C(X)$ is injective on $u+\F_q$ for any $u\in\F_{q^3}$.
\end{itemize}
\end{lemma}

\begin{proof}
Let $\Gamma$ be the set of roots of $X^{q^2}+X^q+X$. It is clear that $X^{q-1}$ induces a surjective map from $\F_{q^3}$ onto $\Gamma$, each of whose fibers has the form $u+\F_q$ with $u\in\F_{q^3}$. Write $g(X)\colonequals L(X) + (X^q-X)\circ B(X)$. Since $(X^q-X)\circ L(X) = L(X)\circ (X^q-X)$ and $(X^q-X)\circ (X^{q^2}+X^q+X) = X^{q^3}-X$, it follows that $(X^q-X)\circ f(X) = g(X)\circ (X^q-X)$ as maps on $\F_{q^3}$. Hence $f(X)$ permutes $\F_{q^3}$ if and only if $g(X)$ permutes $\Gamma$ and $f(X)$ is injective on on $u+\F_q$ for any $u\in\F_{q^3}$. Suppose $u\in\F_{q^3}$, since $X^q-X$ is identical on $u+\F_q$, we know $f(X)$ is injective on $u+\F_q$ if and only if $L(X) + (X^{q^2}+X^q+X) \circ C(X)$ is injective on $u+\F_q$.
\end{proof}

The next result gives a choice of $L(X)$ and $C(X)$ which satisfies the second bullet item of Lemma~\ref{step1}. We will use it in the proof of Theorem~\ref{main1}.

\begin{lemma}\label{fiber}
Assume $q=2^k$ and $Q=2^{\ell}$ for some integers $k>0$ and $\ell\ge 0$. Then $(X^{q^2}+X^q+X)\circ X^{Q+1}$ is injective on $u+\F_q$ for any $u\in\F_{q^3}$ if and only if $\gcd(q-1,Q+1)=1$, or equivalently $\ord_2(k)\le\ord_2(\ell)$.  
\end{lemma}

\begin{proof}
Let us fix an element $u\in\F_{q^3}$. For any $z\in\F_q$, we have 
\[
(u+z)^{Q+1}=u^{Q+1}+u^Qz+uz^Q+z^{Q+1},  
\]
so that we can compute
\begin{align*}
& \;\;\;\;\; (u+z)^{q^2(Q+1)}+(u+z)^{q(Q+1)}+(u+z)^{Q+1} \\
&= (u^{q^2(Q+1)}+u^{q(Q+1)}+u^{Q+1}) + (u^{q^2Q}+u^{qQ}+u^{Q})z + (u^{q^2}+u^{q}+u)z^Q + z^{Q+1} \\
&= (u^{q^2(Q+1)}+u^{q(Q+1)}+u^{Q+1}) + (u^{q^2}+u^q+u)^Qz + (u^{q^2}+u^{q}+u)z^Q + z^{Q+1} \\
&= (u^{q^2(Q+1)}+u^{q(Q+1)}+u^{Q+1}) + (u^{q^2}+u^q+u)^{Q+1} + \bigl( (u^{q^2}+u^q+u) +z \bigr)^{Q+1}.
\end{align*}
Since $u\in\F_{q^3}$ we have $u^{q^2}+u^q+u\in\F_q$. It follows that $(X^{q^2}+X^q+X)\circ X^{Q+1}$ is injective on $u+\F_q$ if and only if $X^{Q+1}$ permutes $\F_q$, or equivalently $\gcd(q-1,Q+1)=1$, i.e.\ $\ord_2(k)\le\ord_2(\ell)$. This concludes the proof.
\end{proof}

The following result is a variant of a procedure introduced in \cite{Z-Redei}:

\begin{lemma}\label{step2}
Suppose $r>0$ is an integer and $B(X)\in\F_{q^3}[X]$. Then the following are equivalent:
\begin{enumerate}
\item $X^r B(X^{q-1})$ permutes $\Gamma$;
\item $\gcd(r,q-1)=1$ and $X^r B(X)^{q-1}$ permutes $\Lambda$;
\item $\gcd(r,q-1)=1$, $B(X)$ has no roots in $\Lambda$, and the rational function $X^r B^{(q)}(-X^{-1}-1) / B(X)$ permutes $\Lambda$.
\end{enumerate}
\end{lemma}

\begin{proof}
It is clear that $X^{q-1}$ induces a surjective map from $\Gamma^*$ onto $\Lambda$, each of whose fibers has the the form $v\F_q^*$ with $v\in\Gamma^*$. Thus $X^r B(X^{q-1})$ permutes $\Gamma$ if and only if $X^r B(X)^{q-1}$ permutes $\Lambda$ and $X^r B(X^{q-1})$ is injective on $v\F_q^*$ for any $v\in\Gamma^*$. We may assume $B(X)$ has no roots in $\Lambda$, since it is implied by each item. For any $v\in\Gamma^*$, since $B(X^{q-1})$ is identically nonzero on $v\F_q^*$, it follows that $X^r B(X^{q-1})$ is injective on $v\F_q^*$ if and only if $X^r$ permutes $\F_q^*$, i.e.\ $\gcd(r,q-1)=1$. It follows that $X^r B(X^{q-1})$ is injective on $v\F_q^*$ for any $v\in\Gamma^*$ if and only if $\gcd(r,q-1)=1$. For any $z\in\Lambda$, we have $z^{q+1}+z+1=0$, which says $z^q=-z^{-1}-1$. We compute
\[
z^r B(z)^{q-1} = z^r B^{(q)}(z^q) / B(z) = z^r B^{(q)}(-z^{-1}-1) / B(z).
\]
It follows that $X^r B(X)^{q-1}$ permutes $\Lambda$ if and only if $X^r B^{(q)}(-X^{-1}-1) / B(X)$ permutes $\Lambda$. This concludes the proof.
\end{proof}

Next let us show Theorem~\ref{key} as follows.

\begin{proof}[Proof of Theorem~\ref{key}]
Since $(X^q-X)\circ X^{Q+1} = X^{Q+1}\bigl(X^{(q-1)(Q+1)}-1\bigr)$, by Lemma~\ref{step2} we know that $(X^q-X)\circ X^{Q+1}$ permutes $\Gamma$ if and only if $\gcd(q-1,Q+1)=1$, $X^{Q+1}-1$ has no roots in $\Lambda$, and 
\[
h(X)\colonequals \frac{X^{Q+1}\bigl((-X^{-1}-1)^{Q+1}-1\bigr)}{X^{Q+1}-1} = \frac{X^Q+X+1}{X^{Q+1}-1} 
\]
permutes $\Lambda$. It is enough to show that if $\gcd(q-1,Q+1)=1$ then $X^{Q+1}-1$ has no roots in $\Lambda$ and $h(X)$ permutes $\Lambda$. Henceforth we suppose that $\gcd(q-1,Q+1)=1$, which by Lemma~\ref{gcd} says that $p=2$ and $\ord_2(k)\le\ord_2(\ell)$. It follows that $\ord_2(3k)\le\ord_2(\ell)$, which by Lemma~\ref{gcd} again says $\gcd(q^3-1,Q+1)=1$, so that $X^{Q+1}$ permutes $\F_{q^3}$. If $X^{Q+1}-1$ has a root $v$ in $\Lambda$, then since $\Lambda\subseteq\F_{q^3}$ we have $v\in\F_{q^3}\cap\mu_{Q+1}$, so that $v=1$, which is impossible since $1\notin\Lambda$. It remains only to show that $h(X)$ permutes $\Lambda$. Since $X^{-1}+1$ permutes $\Lambda$, it reduces to show $\widetilde h(X)$ permutes $\Lambda$, where
\[
\widetilde h(X)\colonequals (X^{-1}+1)\circ h(X) = \frac{X^{Q+1}+X^Q+X}{X^Q+X+1}.
\]

We claim that that $\widetilde h(X)$ permutes $\bP^1(\F_{q^3})$, which concludes the proof. Pick $\omega\in\mybar\F_q$ with $\omega^2+\omega+1=0$. Write $\rho(X)\colonequals (\omega X+1)/(X+\omega)$. By Corollary~\ref{deg1} we know that $\rho(X)$ permutes $\bP^1(\F_{q^3})$ if $k$ is even, and $\rho(X)$ swaps $\bP^1(\F_{q^3})$ and $\mu_{q^3+1}$ if $k$ is odd. First we suppose $\ell$ is even, then $Q\equiv 1\pmod 3$. It is routine to verify that
\[
\rho(X)\circ X^{-(Q+1)}\circ\rho(X) = \frac{X^{Q+1}+X^Q+X}{X^Q+X+1} = \widetilde h(X).
\]
If $k$ is even, then both $\rho(X)$ and $X^{-(Q+1)}$ permutes $\bP^1(\F_{q^3})$, so that $\widetilde h(X)$ permutes $\bP^1(\F_{q^3})$. If $k$ is odd, then $\ord_2(3k)<\ord_2(\ell)$, which by Lemma~\ref{gcd} implies $\gcd(q^3+1,Q+1)=1$. It follows that $\rho(X)$ swaps $\bP^1(\F_{q^3})$ and $X^{-(Q+1)}$ permutes $\mu_{q^3+1}$, so that $\widetilde h(X)$ permutes $\bP^1(\F_{q^3})$. Next we suppose $\ell$ is odd, then $k$ is odd since $\ord_2(k)\le\ord_2(\ell)$. It follows that $\ord_2(\ell)=\ord_2(3k)=1$, which by Lemma~\ref{gcd} implies $\gcd(Q-1,q^3+1)=1$, so that $X^{Q-1}$ permutes $\mu_{q^3+1}$. It is routine to verify that
\[
\rho(X)\circ X^{Q-1}\circ\rho(X) = \frac{X^{Q+1}+X^Q+X}{X^Q+X+1} = \widetilde h(X).
\]
It follows that $\widetilde h(X)$ permutes $\bP^1(\F_{q^3})$, since $\rho(X)$ swaps $\bP^1(\F_{q^3})$ and $\mu_{q^3+1}$, and $X^{Q-1}$ permutes $\mu_{q^3+1}$. This concludes the proof.  
\end{proof}

Now we are ready to show Theorem~\ref{main1} as follows.

\begin{proof}[Proof of Theorem~\ref{main1}]
By Lemma~\ref{step1} we know that $f(X)$ permutes $\F_{q^3}$ if and only if $(X^q-X)\circ X^{Q+1}$ permutes $\Gamma$ and $(X^{q^2}+X^q+X)\circ X^{R+S}$ is injective on $u+\F_q$ for any $u\in\F_{q^3}$. By Theorem~\ref{key} we know that $(X^q-X)\circ X^{Q+1}$ permutes $\Gamma$ if and only if $\gcd(q-1,Q+1)=2$, which by Lemma~\ref{gcd} says that $p=2$ and $\ord_2(k)\le\ord_2(\ell)$. Henceforth we can suppose $p=2$. By Lemma~\ref{fiber} we know that $(X^{q^2}+X^q+X)\circ X^{R+S}$ is injective on $u+\F_q$ for any $u\in\F_{q^3}$ if and only if $\gcd(q-1,R+S)=1$, i.e.\ $\ord_2(k)\le\ord_2(m-n)$. Thus $f(X)$ permutes $\F_{q^3}$ if and only if $\gcd\bigl(q-1,(Q+1)(R+S)\bigr)=1$, or equivalently $p=2$ and $\ord_2(k)\le\min\bigl(\ord_2(\ell),\ord_2(m-n)\bigr)$.
\end{proof}

Next we show Corollary~\ref{main2} as follows.

\begin{proof}[Proof of Corollary~\ref{main2}]
Suppose $q=p^k$ and $Q=p^{\ell}$ for some prime $p$ and some integers $k>0$ and $\ell\ge 0$. We compute 
\[
X^{Q+1}\circ (X^q-X) = X^{qQ+q}-X^{qQ+1}-X^{Q+q}+X^{Q+1}.
\]
Let us apply Theorem~\ref{main1} by taking various choices of $c$, $R$, and $S$.

First, by taking $c=-1$, $R=Q$, and $S=1$ in Theorem~\ref{main1}, we know that 
\[
\widetilde f_1(X)\colonequals -X^{qQ+1}-X^{Q+q} - X^{q^2Q+q^2} 
\]
permutes $\F_{q^3}$ if and only if $\gcd(q-1,Q+1)=1$. Since $f_1(X) \equiv - \widetilde f_1(X)^q \pmod{X^{q^3}-X}$, we know that $f_1(X)$ and $-\widetilde f_1(X)^q$ induce the same map on $\F_{q^3}$. It follows that $f_1(X)$ permutes $\F_{q^3}$ if and only if $\gcd(q-1,Q+1)=1$.

Next, by taking $c=1$, $R=qQ$, and $S=1$ in Theorem~\ref{main1}, we know that 
\[
f_3(X) = X^{qQ+q}-X^{Q+q}+X^{Q+1}+X^{q^2Q+q}+X^{Q+q^2}
\]
permutes $\F_{q^3}$ if and only if $\gcd\bigl(q-1,(Q+1)(qQ+1)\bigr)=1$, or equivalently $\gcd(q-1,Q+1)=1$. So we know that $f_3(X)$ permutes $\F_{q^3}$ if and only if $\gcd(q-1,Q+1)=1$. By replacing $Q$ with $qQ$ in the above result, we know
\[
\widetilde f_2(X) \colonequals X^{q^2Q+q}-X^{qQ+q}+X^{qQ+1}+X^{q^3Q+q}+X^{qQ+q^2}
\]
permutes $\F_{q^3}$ if and only if $\gcd(q-1,qQ+1)=1$, i.e.\ $\gcd(q-1,Q+1)=1$. It is easy to check that $f_2(X)^q \equiv \widetilde f_2(X) \pmod{X^{q^3}-X}$. It follows that $f_2(X)$ permutes $\F_{q^3}$ if and only if $\gcd(q-1,Q+1)=1$. Similarly, by replacing $Q$ with $qQ$ in the above result, we know that
\[
\widetilde f_4(X) \colonequals -X^{qQ+1}+X^{qQ+q}+X^{qQ+q^2}+X^{q^2Q+1}+X^{q^3Q+1}
\]
permutes $\F_{q^3}$ if and only if $\gcd(q-1,qQ+1)=1$, i.e.\ $\gcd(q-1,Q+1)=1$. It is easy to check that $f_4(X)^q \equiv \widetilde f_4(X) \pmod{X^{q^3}-X}$. It follows that $f_4(X)$ permutes $\F_{q^3}$ if and only if $\gcd(q-1,Q+1)=1$. 
\end{proof}

Next let us prove Theorem~\ref{main3} as follows.

\begin{proof}[Proof of Theorem~\ref{main3}]
By Lemma~\ref{step1} we know that $f(X)$ permutes $\F_{q^3}$ if and only if $(X^q-X)\circ X^{Q+1}$ permutes $\Gamma$ and $(X^{q^2}+X^q+X)\circ cX^R$ is injective on $u+\F_q$ for any $u\in\F_{q^3}$. By Theorem~\ref{key} we know that $(X^q-X)\circ X^{Q+1}$ permutes $\Gamma$ if and only if $\gcd(q-1,Q+1)=2$, which by Lemma~\ref{gcd} says that $p=2$ and $\ord_2(k)\le\ord_2(\ell)$. For any $u\in\F_{q^3}$, since it is additive, $(X^{q^2}+X^q+X)\circ cX^R$ is injective on $u+\F_q$ if and only if $(X^{q^2}+X^q+X)\circ cX^R$ is injective on $\F_q$, which is equivalent to that $(c^{q^2}+c^q+c)X^R$ is injective on $\F_q$, which holds if and only if $c^{q^2}+c^q+c\ne 0$, i.e.\ $\Tr_{\F_{q^3}/\F_q}(c)\ne 0$. Thus $f(X)$ permutes $\F_{q^3}$ if and only if $\gcd(q-1,Q+1=1$ and $\Tr_{\F_{q^3}/\F_q}(c)\ne 0$, or equivalently $p=2$, $\ord_2(k)\le\ord_2(\ell)$, and $c^{q^2}+c^q+c\ne 0$.
\end{proof}

Corollary~\ref{main4} follows directly from Theorem~\ref{main3} by taking $c=1$.

Next let us show Theorem~\ref{key-deg3} as follows.

\begin{proof}[Proof of Theorem~\ref{key-deg3}]
Since $(X^q-X)\circ X^3 = X^3\bigl(X^{3(q-1)}-1\bigr)$, by Lemma~\ref{step2} we know that $(X^q-X)\circ X^3$ permutes $\Gamma$ if and only if $\gcd(3,q-1)=1$, $X^3-1$ has no roots in $\Lambda$, and 
\[
h(X)\colonequals \frac{X^3\bigl((-X^{-1}-1)^3-1\bigr)}{X^3-1}
\]
permutes $\Lambda$. We compute
\[
h(X) = -\frac{2X^3+3X^2+3X+1}{X^3-1} = -\frac{2X+1}{X-1},
\]
so that $h(X)$ has degree $1$ if $q\not\equiv 0\pmod 3$. If $X^3-1$ has no roots in $\Lambda$, then $(X^q-X)\circ X^3 = X^3\bigl(X^{3(q-1)}-1\bigr)$ maps $\Gamma^*$ into $\Gamma^*$, which implies $(X^q-X)\circ X^3$ maps $\Lambda$ into $\Lambda$, so that $h(X)$ maps $\Lambda$ into $\Lambda$.

We claim that $X^3-1$ has no roots in $\Lambda$ if and only if $q\equiv 2\pmod 3$, which concludes the proof since if $q\equiv 2\pmod 3$ then $\deg(h)=1$ and $h(X)$ maps $\Lambda$ into $\Lambda$. If $q\equiv 0\pmod 3$ then $1\in\Lambda$ is a root of $X^3-1$. If $q\equiv 1\pmod 3$ then there is $\omega\in\F_q^*$ with $\omega^3=1$ but $\omega\ne 1$, so that $\omega^2+\omega+1=0$. Since $q\equiv 1\pmod 3$, we have $\omega^{q+1}=\omega^2$, thus $\omega^{q+1}+\omega+1=\omega^2+\omega+1=0$, which says $\omega\in\Lambda$ is a root of $X^3-1$. Henceforth we suppose $q\equiv 2\pmod 3$. If there exists $\lambda\in\Lambda$ such that $\lambda$ is a root of $X^3-1$, then $\lambda^{q+1}+\lambda+1=0$ and $\lambda^3=1$, which since $q\equiv 2\pmod 3$ implies $0=\lambda^{q+1}+\lambda+1=\lambda+2$, i.e.\ $\lambda=-2$. Since $\lambda^3=1$, it follows that $1=\lambda^3=(-2)^3=-8$, which says $q=0$, hence $p=3$, which contradicts with $q\equiv 2\pmod 3$.
\end{proof}

Next let us prove Theorem~\ref{main-deg3} as follows.

\begin{proof}[Proof of Theorem~\ref{main-deg3}]
By Lemma~\ref{step1} we know that $f(X)$ permutes $\F_{q^3}$ if and only if $(X^q-X)\circ X^3$ permutes $\Gamma$ and $(X^{q^2}+X^q+X)\circ cX^Q$ is injective on $u+\F_q$ for any $u\in\F_{q^3}$. By Theorem~\ref{key-deg3} we know that $(X^q-X)\circ X^3$ permutes $\Gamma$ if and only if $q\equiv 2\pmod 3$. For any $u\in\F_{q^3}$, since it is additive, $(X^{q^2}+X^q+X)\circ cX^Q$ is injective on $u+\F_q$ if and only if $(X^{q^2}+X^q+X)\circ cX^Q$ is injective on $\F_q$, which is equivalent to that $(c^{q^2}+c^q+c)X^Q$ is injective on $\F_q$, which holds if and only if $c^{q^2}+c^q+c\ne 0$, i.e.\ $\Tr_{\F_{q^3}/\F_q}(c)\ne 0$. Thus $f(X)$ permutes $\F_{q^3}$ if and only if $q\equiv 2\pmod 3$ and $\Tr_{\F_{q^3}/\F_q}(c)\ne 0$, or equivalently $p=2$, $k$ is odd, and $c^{q^2}+c^q+c\ne 0$.
\end{proof}

Corollary~\ref{main2-deg3} follows directly from Theorem~\ref{main-deg3} by taking $c\in\F_q^*$.




\begin{thebibliography}{9}
\newcommand{\au}[1]{{#1},}
\newcommand{\ti}[1]{\textit{#1},}
\newcommand{\jo}[1]{{#1}}
\newcommand{\vo}[1]{\textbf{#1}}
\newcommand{\yr}[1]{(#1),}
\newcommand{\ppx}[1]{#1,}
\newcommand{\pp}[1]{#1.}
\newcommand{\pps}[1]{#1}
\newcommand{\bk}[1]{{#1},}
\newcommand{\inbk}[1]{in: \bk{#1}}
\newcommand{\xxx}[1]{{arXiv:#1}.}

\bibitem{DXZ}
\au{Z. Ding, W. Xiong, and Q. Zhang}
\ti{Exceptional extensions of local fields and the Carlitz-Wan conjecture}
\jo{Sci. China Math.}
\vo{68}
\yr{2025}
\pp{2815--2826}

\bibitem{DZ-quad}
\au{Z. Ding and M. E. Zieve}
\ti{Determination of a class of permutation quadrinomials}
\jo{Proc. London Math. Soc. (3)}
\vo{127}
\yr{2023}
\pp{221--260}

\bibitem{Gologlu-fractional}
\au{F. G\"olo\u{g}lu}
\ti{Classification of fractional projective permutations over finite fields}
\jo{Finite Fields Appl.}
\vo{81}
\yr{2022}
Paper No. 102027, 50 pp.

\bibitem{GRZ}
\au{R. M. Guralnick, J. Rosenberg and M. E. Zieve}
\ti{A new family of exceptional polynomials in characteristic two}
\jo{Ann. of Math. (2)}
\vo{172}
\yr{2010}
\pp{1361--1390}

\bibitem{GZ}
\au{R. M. Guralnick and M. E. Zieve}
\ti{Polynomials with $\PSL(2)$ monodromy}
\jo{Ann. of Math. (2)}
\vo{172}
\yr{2010}
\pp{1315--1359}

\bibitem{ZZWPL}
\au{T. Zhang, L. Zheng, H. Wang, J. Peng, and Y. Li}
\ti{Further results on permutation pentanomials over $\F_{q^3}$ in characteristic two}
\jo{Finite Fields Appl.}
\vo{110}
\yr{2026}
Paper No. 102743, 26 pp.

\bibitem{Zlem}
\au{M. E. Zieve}
\ti{Some families of permutation polynomials over finite fields}
\jo{Int. J. Number Theory}
\vo{4}
\yr{2008}
\pp{851--857}

\bibitem{Z-Redei}
\au{M. E. Zieve}
\ti{Permutation polynomials on\/ $\F_q$ induced from R\'edei function
bijections on subgroups of\/ $\F_q^*$}
\jo{Monatsh. Math.}, to appear.
arXiv:1310.0776v2, 7 Oct 2013.

\bibitem{Z-exc}
\au{M. E. Zieve}
\ti{Exceptional polynomials}
\inbk{Handbook of Finite Fields}
CRC Press, Boca Raton, FL
\yr{2013}
\pp{229--233}


\end{thebibliography}
\end{document}